\documentclass[a4paper,12pt]{article}
\usepackage{times}
\usepackage[english]{babel}
\usepackage{amssymb,amsmath}
\numberwithin{equation}{section}
\usepackage{amsthm}
\usepackage{array}
\usepackage{wasysym}
\usepackage[noadjust]{cite}
\usepackage{hyperref}
\usepackage[height=22.7cm , width = 16cm , top = 4cm , left = 3cm, a4paper]{geometry}
\usepackage{amssymb}
\usepackage{amsmath}
\usepackage{cite}
 \usepackage{pstricks}
\usepackage{epsfig}
\usepackage{verbatim}
\usepackage{multicol}
\usepackage{graphicx, color, psfrag}
\usepackage{graphics, psfrag}
\usepackage[numbers,sort]{natbib}
 \newtheorem{theorem}{Theorem}[section]

\theoremstyle{definition}

\newcommand{\e}{\end{document}}

\begin{document}

\thispagestyle{empty}

\author{
{{\bf  M. A. El-Damcese$^1$, Abdelfattah Mustafa$^{2,}$\footnote{Corresponding author: abdelfatah\_mustafa@yahoo.com}, }} \\
{{ \bf B. S. El-Desouky$^2$ \; and   M. E. Mustafa$^2$} }
 { }\vspace{.2cm}\\
 \small \it $^{1}$Tanta University, Faculty of Science, Mathematics Department, Egypt.\\
 \small \it $^{2}$Department of Mathematics, Faculty of Science, Mansoura University, Mansoura 35516, Egypt.
}

\title{The Odd Generalized Exponential Linear Failure Rate Distribution}

\date{}

\maketitle
\small \pagestyle{myheadings}
        \markboth{{\scriptsize The Odd Generalized Exponential Linear Failure Rate Distribution }}
        {{\scriptsize {M.A. El-Damcese, Abdelfattah Mustafa, B.S. El-Desouky and   M.E. Mustafa}}}

\hrule \vskip 8pt
\begin{abstract}
In this paper we propose a new lifetime model, called the odd generalized exponential linear failure rate distribution. Some statistical properties of the proposed distribution such as the moments, the quantiles, the median, and the mode are investigated. The method of maximum likelihood is used for estimating the model parameters. An applications to real data is carried out to illustrate that the new distribution is more flexible and effective than other popular distributions in modeling lifetime data.
\end{abstract}

\noindent
{\bf Keywords:}
{\it Hazard function; Moments; Maximum likelihood estimation; Linear failure rate distribution.}

\noindent
{\bf 2010 MSC:} {\em  60E05,  62F10, 62F12, 62N02, 62N05.}

\section{Introduction}
Some distributions such as the exponential (E), Rayleigh (R), generalized exponential (GE), and linear exponential (LE) are used for modelling the lifetime data in reliability. These distributions have several desirable properties and satisfactory interpretations which enable them to be used frequently. It is well-known that the exponential distribution can have only constant hazard rate function whereas, Rayleigh, linear failure rate, and generalized exponential distributions can have only monotone (increasing in case of linear failure rate distribution and increasing/decreasing in case of generalized exponential distribution) failure rate functions. However, the above distributions sometimes have some respective drawbacks in analyzing lifetime data. Gupta and Kundu \cite{5} proposed a generalization of the exponential distribution named as Generalized Exponential (GE) distribution. The two-parameter GE distribution with parameters $ \alpha > 0$ and $\beta > 0$, has the following distribution function
\begin{equation} \label{1-1}
F(x)=\left[ 1-e^{-\alpha x}\right] ^{\beta },\;x>0,\;\text{\/}\alpha >0,\;%
\text{\/}\beta >0.
\end{equation}

\noindent
The linear exponential (LE) distribution is also known as the Linear Failure Rate (LFR) distribution, having exponential and Rayleigh distributions as special cases, Bain \cite{2}. The two-parameter LE distribution with parameters $ a > 0$ and $ b> 0$, has the following distribution function
\begin{equation} \label{1-2}
F(x)=1-e^{-ax-\frac{b}{2}x^{2}},\;x>0,\;\text{\/}a>0,\;\text{\/}b>0.
\end{equation}

\noindent
Sarhan and Kundu \cite{10} presented a three-parameter generalized linear failure rate (GLFR) distribution by exponentiating the LFR distribution as was done for the exponentiated Weibull distribution by Mudholkar et al. \cite{8}. The exponentiation introduces an extra shape parameter in the model, which may yield more flexibility in the shape of the probability density function (pdf) and hazard function. The distribution  function of the generalized linear failure rate (GLFR) distribution is given as
\begin{equation} \label{1-3}
F(x)=\left[ 1-e^{-ax-\frac{b}{2}x^{2}}\right] ^{\beta },\;x>0,\;\text{\/}%
a>0,\;\text{\/}b>0,\;\text{\/}\beta >0.
\end{equation}

\noindent
It is observed that the GLFR distribution has decreasing or unimodal pdf and it can have increasing, decreasing, and bathtub-shaped hazard functions. Another important characteristic of GLFR distribution is that it contains, as special sub-models, the generalized exponential (GE), generalized Rayleigh (GR), Linear failure rate (LFR), exponential (E), and Rayleigh (R) distributions, \cite{4,10}.  Jamkhaneh \cite{6} introduced four-parameter distribution called the modified generalized linear failure rate (MGLFR) distribution. Mahmoud and Alam \cite{9} proposed a generalization of linear exponential distribution called the generalized linear exponential (GLE) distribution. Anew four-parameter generalization of the linear failure rate (LFR) distribution which is called Beta-linear failure rate (BLFR) distribution is introduced by Jafari and Mahmoudi \cite{7}. The BLFR distribution is quite flexible and can be used effectively in modeling survival data and reliability problems. It can have a constant, decreasing, increasing, upside-down bathtub (unimodal) and bathtub-shaped failure rate function depending on its parameters, and it also includes some well-known lifetime distributions as special sub-models. Another generalized version of linear exponential distribution introduced by Yuzhu tiana et al. \cite{11} called the new generalized linear exponential (NGLE) distribution and discuss some of its properties, it also includes some well-known lifetime distributions as special sub-models. Yuzhu tiana et al. \cite{12} also presented another generalization of linear exponential distribution called the transmuted linear exponential (TLE) distribution. Recently, a new class of univariate continuous distributions called the odd generalized exponential (OGE) class introduced by \cite{3,13}. This class is flexible because of the hazard rate shapes could be increasing, decreasing, bathtub and upside down bathtub. The odd generalized exponential (OGE) class is defined as follows. If $G(x)$, $x>0$ is cumulative distribution function (cdf) of a random variable X, then the corresponding survival function is $\overline{ G} (x) =1-G(x)$ and the probability density function is $g(x)$, then we define the cdf of the OGE class by replacing $x$ in the distribution function of generalized exponential (GE) distribution given in equation (\ref{1-1}) by $\frac{G(x)}{\overline{G}(x)}$ leading to
\begin{equation} \label{1-4}
F(x)=\left[ 1-e^{-\alpha \frac{G(x)}{\overline{G}(x)} } \right]
^{\beta }, \;x > 0,\; \alpha >0, \; \beta >0.
\end{equation}

\noindent
The probability density function corresponding to (\ref{1-4}) is given by
\begin{equation} \label{1-5}
f(x)=\frac{\alpha \beta g(x)}{\overline{ G}{(x)}^2} \;e^{-\alpha \frac{ G(x)}{\overline{G}(x)}}
\left[ 1-e^{-\alpha \frac{G(x)}{\overline{G}(x)}}\right]^{\beta -1}, \; x>0,\; \alpha >0,\; \beta >0.
\end{equation}

\noindent
In this article we present a new distribution depending on Linear Failure Rate distribution called the Odd Generalized Exponential-Linear Failure Rate (OGE-LFR) distribution by using the class of univariate distributions defined above.\\

\noindent
This paper is organized as follows. In Section 2 we define the cumulative distribution function, density function, reliability function, hazard function and the reversed hazard function of the odd generalized exponential-linear failure rate (OGE-LFR) distribution. In Section 3 we study some different properties of (OGE-LFR) distribution include, the quantile function, median, mode, and the moments. Section 4 discusses the distribution of the order statistics for (OGE-LFR) distribution. Moreover, maximum likelihood estimation of the parameters is determined in Section 5. Finally, an application of OGE-LFR distribution using a real data set is presented in Section 6.


\section{The OGE-LFR distribution}
\subsection{OGE-LFR specifications}
In this subsection we present a new four parameters distribution called Odd Generalized Exponential-Linear Failure Rate (OGE-LFR) distribution with parameters $\alpha ,a,b,$ and $ \beta$  written as OGE-LFR($\Psi)$, where the vector $\Psi $ is defined in the form $\Psi =( \alpha,a, b, \beta )$.\\
A random variable $X$ is said to have OGE-LFR with parameters $\alpha ,a,b,$  and $\beta$ if its cumulative distribution function (cdf) given as
\begin{equation} \label{2-1}
F(x)=\left[ 1-e^{-\alpha \left( e^{ax+\frac{b}{2}x^2}-1\right) }\right]^{\beta }, \; x>0,\; \alpha, a, b, \beta >0.
\end{equation}

\noindent
The corresponding pdf has the form
\begin{equation} \label{2-2}
f(x)=\alpha \beta (a+bx) e^{ax+\frac{b}{2}x^{2}} e^{-\alpha \left( e^{ax+ \frac{b}{2}x^{2}}-1\right) }\left[ 1-e^{-\alpha \left(
e^{ax+\frac{b}{2}x^{2}}-1\right) }\right]^{\beta -1} ,
\end{equation}
where $\;x>0,\; \alpha, a, b, \beta >0.$

\subsection{Survival and hazard functions}
If a random variable X has cdf in (\ref{2-1}), then the corresponding
survival function is given by
\begin{equation} \label{2-3}
S(x)=1-\left[ 1-e^{-\alpha \left( e^{ax+\frac{b}{2}x^{2}}-1\right) }\right]
^{\beta }.
\end{equation}

\noindent
The hazard function of OGE-LFR($\Psi $) is defined as follow
\begin{equation} \label{2-4}
h(x)=\frac{\alpha \beta (a+bx) e^{ax+\frac{b}{2}x^2} e^{-\alpha \left(e^{ax+\frac{b}{2}x^{2}}-1\right) } \left[ 1-e^{-\alpha \left( e^{ax+\frac{b}{2}x^{2}}-1\right) }\right]^{\beta -1}}{1-\left[ 1-e^{-\alpha \left( e^{ax+\frac{b}{2}x^2}-1\right) }\right] ^{\beta }}.
\end{equation}
Also the reversed hazard function of OGE-LFR($\Psi $) is given as follow
\begin{equation} \label{2-5}
r(x)=\frac{\alpha \beta (a+bx) e^{ax+\frac{b}{2}x^2}}{e^{\alpha \left(
e^{ax+\frac{b}{2}x^2}-1\right) }-1}.
\end{equation}

\begin{figure}[htp]
\centering
\includegraphics[width=0.47\textwidth]{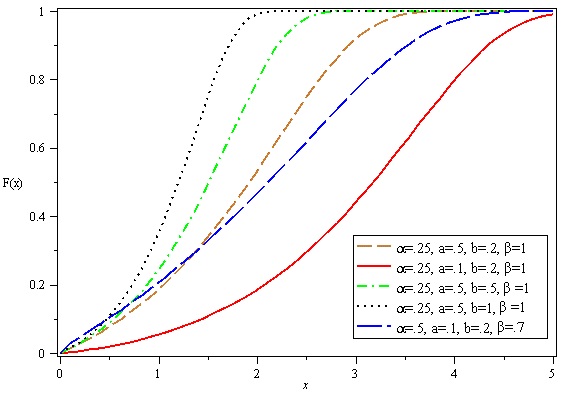}
\includegraphics[width=0.49\textwidth]{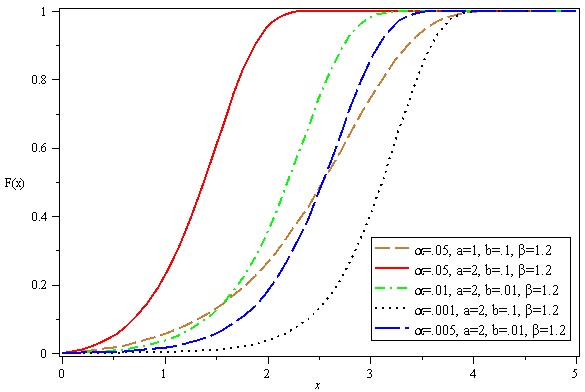}
\caption{The CDF of various OGE-LFR distribution for different values of parameters.}
\label{fig1}
\end{figure}

\begin{figure}[htp]
\centering
\includegraphics[width=0.57\textwidth]{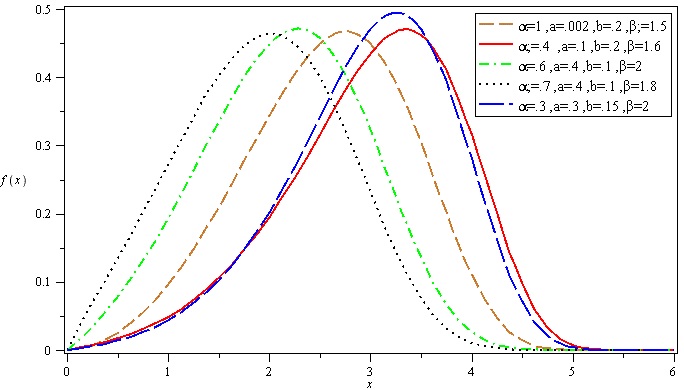}
\includegraphics[width=0.38\textwidth]{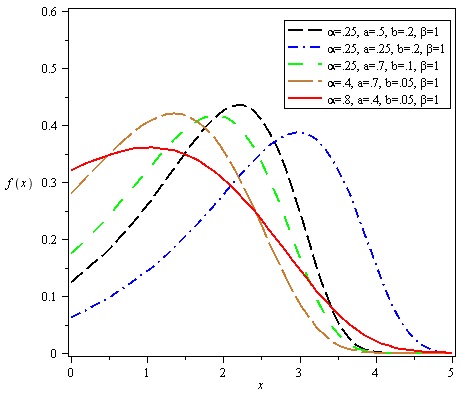}
\caption{The PDF of various OGE-LFR distribution for different values of parameters.}
\label{fig2}
\end{figure}

\begin{figure}[htp]
\centering
\includegraphics[width=0.49\textwidth]{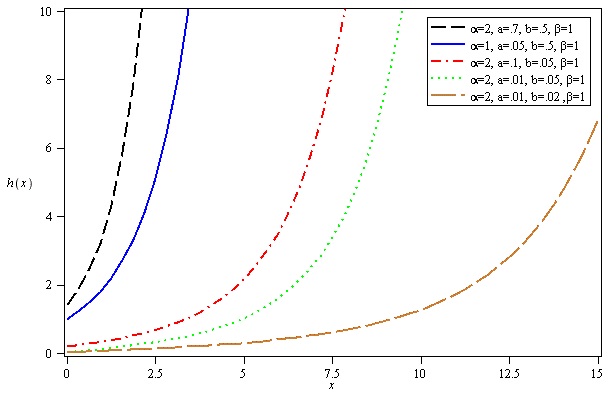}
\includegraphics[width=0.47\textwidth]{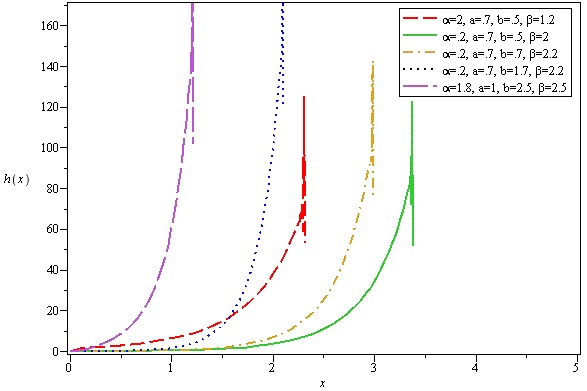}
\caption{The Hazard function of various OGE-LFR distribution for different values of parameters.}
\label{fig3}
\end{figure}

\begin{figure}[htp]
\centering
\includegraphics[width=0.49\textwidth]{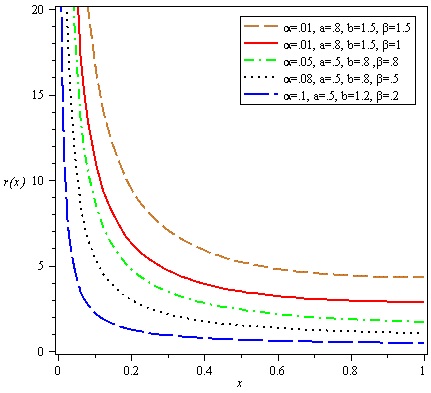}
\includegraphics[width=0.49\textwidth]{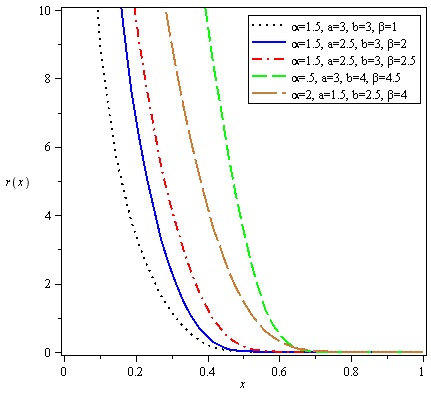}
\caption{The Reversed hazard function of various OGE-LFR distribution for different values of parameters.}
\label{fig4}
\end{figure}

\section{Statistical Properties}
This section is devoted for studying some statistical properties for the odd generalized exponential-linear failure rate (OGE-LFR), specifically quantile function, median and the moments.

\subsection{Quantile, Median and Mode}
The quantile of the OGE-LFR($\Psi $) distribution is simply the solution of the following equation, with respect to $x_q, \; 0 < q < 1$
\begin{equation} \label{3-1}
q=F(x_q)=\left[ 1-e^{-\alpha \left( e^{ax_{q}+\frac{b}{2}x_q^2}-1 \right) }\right]^{\beta }.
\end{equation}

\noindent
By solving equation (\ref{3-1}), we obtain $x_q $ as follow
\[
x_q =\frac{-a \pm \sqrt{a^2+2 b\ln\left\{ \ln \left[ \frac{1}{\left(
1-q^{\frac{1}{b}}\right) ^{\frac{1}{\alpha }}}\right] \right\} }}{b}.
\]
Since the quantile $x_q$ is positive, then we obtain the quantile as follow
\begin{equation} \label{3-3}
x_q=\frac{-a+ \sqrt{a^2+2b\ln \left\{ \ln \left[ \frac{1}{\left( 1-q^{\frac{1}{b}}\right)^{\frac{1}{\alpha }}}\right] \right\} }}{b}.
\end{equation}

\noindent
The median can be derived from (\ref{3-3}) be setting $q=\frac{1}{2}$. That is, the median is given by the following relation
\begin{equation} \label{3-4}
Med=\frac{-a+\sqrt{a^2+2b\ln \left\{ \ln \left[ \frac{1}{\left( 1-(\frac{1
}{2})^{\frac{1}{b}}\right)^{\frac{1}{\alpha }}}\right] \right\} }}{b}.
\end{equation}

\noindent
Moreover, the mode of the OGE-LFR($\Psi $) distribution can be obtained by deriving its pdf with respect to $x$ and equal it to zero. Thus the mode of the OGE-LFR($\Psi $) distribution can be obtained as a nonnegative solution of the following nonlinear equation
\begin{equation} \label{3-5a}
1+b(a+bx)^{-2}-\alpha e^{ax+\frac{b}{2}x^2}\left[ 1-\frac{\beta -1}{e^{\alpha \left( e^{ax+\frac{b}{2}x^{2}}-1\right) }-1}\right] =0.
\end{equation}

\noindent
It is not possible to get an explicit solution of (\ref{3-5a}) in the general case. Numerical methods should be used such as fixed-point or bisection method to solve it.

\subsection{The moments}
In this subsection, we will derive the rth moments of the OGE-LFR($\Psi $) distribution as infinite series expansion.
\begin{theorem}
The rth moment of a random variable $X \sim $OGE-LFR($\Psi $), where $\Psi =(\alpha, a, b, \beta)$ is given by
\begin{eqnarray*}
\mu _{r}^{\acute{}} & = &\sum^{\infty }_{i=0} \sum_{j=0}^{\infty }
\sum^{j}_{k=0} \sum_{L=0}^{\infty } \tbinom{\beta -1}{i} \tbinom{j}{k}
\left( -1\right) ^{i+j+k}\frac{ \beta \alpha ^{j+1}b^{L}(i+1)^{j}}{j!L!2^{L}}
\times \\
&&
\left[ \frac{(r+2L)!}{a^{r+2L}(j-k+1)^{r+2L+1}}+\frac{b(r+2L+1)!}{%
a^{r+2L+2}(j-k+1)^{r+2L+2}}\right].
\end{eqnarray*}
\end{theorem}
\begin{proof}
The rth moment of a random variable $X$ with pdf $f(x)$ is defined by

\begin{equation} \label{3-5}
\mu_{r}^{{\acute{}}}=\int^{\infty }_{0} x^{r} f(x) dx.
\end{equation}

\noindent
Substituting from (\ref{2-2}) into (\ref{3-5}), we obtain
\begin{equation} \label{3-6}
\mu _{r}^{{\acute{}}}=\int^{\infty }_{0}x^{r} \alpha \beta (a+bx)e^{ax+\frac{b}{2}x^2}e^{-\alpha \left( e^{ax+\frac{b}{2}x^{2}}-1\right) }\left[
1-e^{-\alpha \left(e^{ax+\frac{b}{2}x^{2}}-1\right) }\right]^{\beta -1}dx.
\end{equation}

\noindent
Sinec $ 0< \left[ 1-e^{-\alpha \left( e^{ax+\frac{b}{2}x^{2}}-1\right) }\right] < 1 $ for $ x>0,$ we obtain
\begin{equation} \label{3-7}
\left[ 1-e^{-\alpha \left( e^{ax+\frac{b}{2}x^{2}}-1\right) }\right]^{\beta
-1} = \sum_{i=0}^{\infty } \tbinom{\beta -1}{i}\left(
-1\right)^{i}e^{-\alpha i\left( e^{ax+\frac{b}{2}x^{2}}-1\right)}.
\end{equation}

\noindent
Substituting from (\ref{3-7}) into (\ref{3-6}), we get
\begin{equation*}
\mu _{r}^{{\acute{}}}= \sum_{i=0}^{\infty } \tbinom{\beta -1}{i}\left(
-1\right)^{i} \alpha \beta \int^{\infty }_{0}
x^{r}(a+bx)e^{ax+\frac{b}{2}x^{2}}e^{-\alpha (i+1)\left( e^{ax+\frac{b}{2}%
x^{2}}-1\right)}dx.
\end{equation*}

\noindent
Using series expansion of $e^{-\alpha (i+1)\left( e^{ax+\frac{b}{2}x^{2}}-1
\right) },$ we obtain
\begin{equation*}
\mu _{r}^{{\acute{}}}=\sum_{i=0}^{\infty } \sum_{j=0}^{\infty }
\tbinom{\beta -1}{i} \left( -1\right)^{i+j} \frac{\beta \alpha^{j+1} (i+1)^{j}}{j!} \int^{\infty }_{0} x^{r}(a+bx)e^{ax+ \frac{b}{2}x^{2}}\left[ e^{ax+\frac{b}{2}x^{2}}-1\right] ^{j}dx.
\end{equation*}

\noindent
Using binomial expansion of $\left[ e^{ax+\frac{b}{2}x^{2}}-1\right]^{j},$
we obtain
\begin{eqnarray*}
\mu _{r}^{{\acute{}}} & = & \sum_{i=0}^{\infty } \sum_{j=0}^{\infty } \sum^{j}_{k=0} \tbinom{\beta -1}{i}\tbinom{j}{k} \left( -1\right)^{i+j+k}\frac{\beta \alpha ^{j+1}(i+1)^{j}}{j!} \times \\
& &
\hspace{2.7cm}
\int^{\infty }_{0} x^{r}(a+bx)e^{a(j-k+1)x}\ e^{\frac{b}{2}%
(j-k+1)x^{2}}dx.
\end{eqnarray*}

\noindent
Using series expansion of $e^{\frac{b}{2}(j-k+1)x^2},$ we obtain
\begin{eqnarray*}
&&\left. \mu _{r}^{{\acute{}} }=\sum_{i=0}^{\infty }\sum_{j=0}^{\infty }\sum^{j}_{k=0} \sum_{L=0}^{\infty } \tbinom{\beta -1}{i}\tbinom{j}{k}\left( -1\right)^{i+j+k}\frac{\beta \alpha ^{j+1}b^{L}(i+1)^{j}(j-k+1)^{L}}{j!L!2^{L}}\right. \times \\
&&
\hspace{2cm}
\left[ a \int^{\infty }_{0} x^{r+2L}  e^{a(i-k+1)x}dx+b
\int^{\infty }_{0} x^{r+2L+1}\ e^{a(i-k+1)x}dx \right] .
\end{eqnarray*}

\noindent
By using the definition of gamma function in the form, Zwillinger \cite{14},
\begin{equation*}
\Gamma (z) = x^z \int^{\infty }_{0} e^{tx}\ t^{z-1}dt, \; z, \; x>0.
\end{equation*}

\noindent
Finally, we obtain the rth moment of OGE-LFR in the form
\begin{eqnarray*}
\mu _{r}^{{\acute{}}} & = & \sum_{i=0}^{\infty } \sum_{j=0}^{\infty } \sum^{j}_{k=0} \sum_{L=0}^{\infty } \tbinom{\beta-1}{i}\tbinom{j}{k}(-1 )^{i+j+k}\frac{ \beta \alpha ^{j+1} b^{L}(i+1)^{j}(j-k+1)^{L}}{j!L!2^{L}}
\times \\
&&
\hspace{2cm}
\left[ \frac{(r+2L)!}{a^{r+2L}(j-k+1)^{r+2L+1}}+\frac{b(r+2L+1)!}{a^{r+2L+2}(j-k+1)^{r+2L+2}}\right].
\end{eqnarray*}

\noindent
This completes the proof.
\end{proof}


\section{Order Statistics}
Let X$_{1:n}, X_{2:n}, \cdots, X_{n:n}$ denote the order statistics obtained from a random sample $X_1, X_2, \cdots, X_n$  which taken from a continuous population with cumulative distribution function (cdf) $F(x,\Psi )$ and probability density function (pdf) $f(x,\Psi )$, then the probability density function of $X_{r:n}$ is given by
\begin{equation} \label{4-1}
f_{r:n}(x,\Psi )=\frac{1}{B(r,n-r+1)}\left[ F(x,\Psi )\right] ^{r-1}\left[
1-F(x,\Psi )\right] ^{n-r}f(x,\Psi ),
\end{equation}

\noindent
where $f(x,\Psi ),\ F(x,\Psi )$ are the pdf and cdf of OGE-LFR($\Psi $) distribution given by (\ref{21}) and (\ref{22}) respectively and $B(.,.)$ is the beta function, also we define first order statistics $X_{1:n}= \min(X_1 ,X 2,\cdots, X_n)$, and the last order statistics as $X_{n:n}= \max(X_1 ,X_2, \cdots, X_n)$. Since $0 < F(x,\Psi )< 1$ for $x > 0$, we can use the binomial expansion of $\left[ 1-F(x,\Psi )\right]^{n-r}$\ given as follows
\begin{equation} \label{4-2}
\left[ 1-F(x,\Psi )\right] ^{n-r}=\sum_{i=0}^{n-r} \binom{n-r}{i}(-1)^{i}\left[ F(x,\Psi )\right] ^{i}.
\end{equation}

\noindent
Substituting from (\ref{4-2}) into (\ref{4-1}), we obtain
\begin{equation} \label{4-3}
f_{r:n}(x,\Psi )=\frac{1}{B(r,n-r+1)}f(x;\Psi )\sum_{i=0}^{n-r}
\binom{n-r}{i}(-1)^{i}\left[ F(x,\Psi )\right] ^{i+r-1}.
\end{equation}

\noindent
Substituting from (\ref{2-1}) and (\ref{2-2}) into (\ref{4-3}), we obtain
\begin{equation} \label{4-4}
f_{r:n}(x;\alpha ,a,b,\beta )=\sum_{i=0}^{n-r}
\frac{(-1)^{i}n!}{i!(r-1)!(n-r-i)!(r+i)}f(\alpha ,a,b,(r+i)\beta ).
\end{equation}

\noindent
Relation (\ref{4-4}) shows that $f_{r:n}(x,\Psi )$ is the weighted average of the odd generalized exponential-linear failure rate with different shape parameters.

\section{Estimation and Inference}
Now, we discuss the estimation of the OGE-LFR($\alpha, a, b, \beta $) parameters by using the method of maximum likelihood based on a complete sample.

\subsection{Maximum likelihood estimators}
Let $X_1, X_2, \cdots, X_n$  be a random sample of size n from $X \sim$ OGE-LFR($\alpha,a, b, \beta $) with observed values $x_1, x_2, \cdots, x_n$, then the log-likelihood function can be written as
\begin{equation} \label{5-1}
\mathcal{L }=\prod^n_{i=1} f(x_i; \alpha,a,b,\beta ).
\end{equation}

\noindent
Substituting from (\ref{2-2}) into (\ref{5-1}), we get
\begin{equation*}
\mathcal{L } = \prod^{n}_{i=1} \alpha \beta
\left[ a+bx_{i}\right] e^{ax_{i}+\frac{b}{2}x_{i}^{2}}e^{-\alpha \left[
e^{ax_{i}+\frac{b}{2}x_{i}^{2}}-1\right] }\left[ 1-e^{-\alpha \left[
e^{ax_{i}+\frac{b}{2}x_{i}^{2}}-1\right] }\right] ^{\beta -1}.
\end{equation*}

\noindent
The log-likelihood function can be written as
\begin{eqnarray} \label{5-2}
L & = & n \ln(\alpha )+n \ln (\beta )+\sum^n_{i=1} \ln \left[ a+bx_{i}\right] +\sum^{n}_{i=1} \left[ ax_{i}+\frac{b}{2}x_{i}^{2}\right] -
\nonumber\\
&&
\alpha \sum^n_{i=1} \left[e^{ax_{i}+\frac{b}{2}x_{i}^{2}}-1\right]
+(\beta -1)\sum^n_{i=1} \ln \left[ 1-e^{-\alpha \left( e^{ax_{i}+\frac{b}{2}x_{i}^{2}}-1\right) }\right].
\end{eqnarray}

\noindent
The maximum likelihood estimates of the parameters are obtained by Differentiating the log-likelihood function L with respect to the parameters $\alpha, a, b$ and $\beta$ and setting the result to zero
\begin{eqnarray} \label{5-3}
\frac{\partial L}{\partial \beta }& = &\frac{n}{\beta }+\sum^n_{i=1} \ln \left[ 1-e^{-\alpha \left( e^{ax_{i}+\frac{b}{2}x_{i}^{2}}-1\right) }\right] =0,
\\ \label{5-4}
\frac{\partial L}{\partial \alpha }& = &\frac{n}{\alpha }-\sum^n_{i=1}
\left[ \varphi (x_{i},a,b)-1\right] +(\beta -1)\sum^{n}_{i=1} \frac{\left[ \varphi (x_{i},a,b)-1\right] }{\psi (x_{i},\alpha ,a,b)}=0,
\\ \label{5-5}
\frac{\partial L}{\partial a}& =&\sum^n_{i=1} \frac{1}{a+bx_{i}}+\sum^{n}_{i=1}x_{i}-\alpha \sum^n_{i=1} \varphi (x_{i},a,b)x_{i}+(\beta -1)\alpha \sum^n_{i=1} \frac{\varphi (x_{i},a,b)x_{i}}{\psi (x_{i},\alpha
,a,b)}=0,
\nonumber\\ &&
\end{eqnarray}

\begin{equation} \label{5-6}
\frac{\partial L}{\partial b}=\sum^n_{i=1} \frac{x}{a+bx_{i}}+\frac{1}{2}\sum^n_{i=1} x_{i}^{2}-\frac{\alpha }{2}\sum^n_{i=1} \varphi (x_{i},a,b)x_{i}^{2}+ \frac{(\beta -1)\alpha }{2}\sum^n_{i=1} \frac{\varphi
(x_{i},a,b)x_{i}^{2}}{\psi (x_{i},\alpha ,a,b)}=0,
\end{equation}

\noindent
Where the nonlinear functions $\psi (x_i,\alpha ,a,b)$ and $\varphi (x_i,a,b)$ are given by
\begin{eqnarray*}
\varphi (x_{i},a,b) & = &e^{ax_{i}+\frac{b}{2}x_{i}^{2}},
\\
\psi (x_{i},\alpha ,a,b)& = & e^{\alpha \left[ e^{ax_{i}+\frac{b}{2}x_{i}^{2}}-1\right] }-1.
\end{eqnarray*}

\noindent
From equation (\ref{5-3}), we obtain the maximum likelihood estimate of $\beta$ in a closed form as follow
\begin{equation} \label{5-7}
\hat{\beta}=\frac{-n}{\sum^n_{i=1} \ln \left[
1-e^{-\alpha \left[ e^{ax_{i}+\frac{b}{2}x_{i}^{2}}-1\right] }\right] }.
\end{equation}

\noindent
Substituting from (\ref{5-7}) into (\ref{5-4}), (\ref{5-5}) and (\ref{5-6}), we get the MLEs of $\alpha, a, b$ by solving the following system of non-linear equations
\begin{eqnarray*}
\frac{n}{\hat{\alpha}}-\sum^n_{i=1} \left[ \varphi
(x_{i},\hat{a},\hat{b})-1\right] +(\hat{\beta}-1)\sum^n_{i=1}\frac{\left[ \varphi (x_{i},\hat{a},\hat{b})-1\right] }{\psi (x_{i}, \hat{\alpha},\hat{a},\hat{b})}& = &0,
\\
\sum^n_{i=1}\frac{1}{\hat{a}+\hat{b}x_{i}}+\sum^n_{i=1} x_{i}-\hat{\alpha}\sum^n_{i=1} \varphi (x_{i},\hat{a},\hat{b})x_{i}+(\hat{\beta}-1)\hat{\alpha}\sum^n_{i=1} \frac{\varphi (x_{i},\hat{a},\hat{b})x_{i}}{\psi
(x_{i},\hat{\alpha},\hat{a},\hat{b})} & = &0,
\\
\sum^n_{i=1}\frac{x}{\hat{a}+\hat{b}x_{i}}+\frac{1}{2}
\sum^n_{i=1} x_{i}^{2}-\frac{\hat{\alpha}}{2}\sum^n_{i=1} \varphi (x_{i},\hat{a},\hat{b})x_{i}^{2}+\frac{(\hat{\beta}-1)\hat{\alpha}}{2}\sum^n_{i=1}\frac{\varphi (x_{i},\hat{a},\hat{b})x_{i}^{2}}{\psi (x_{i},\hat{\alpha},\hat{a},%
\hat{b})}& = &0.
\end{eqnarray*}

\noindent
There is no closed form solution to these equations, so statistical software or numerical technique must be applied.

\subsection{Asymptotic confidence bounds}
In this subsection, we derive the asymptotic confidence intervals of the unknown parameters $\alpha, a, b$ and $\beta.$ As the sample size $n\longrightarrow \infty,$ then $(\hat{\alpha}-\alpha,\hat{a}-a,\hat{b}-b,\hat{\beta}-\beta )$ approaches a multivariate normal vector with zero means and the variance $I_{0}^{-1}(\hat{\alpha},\hat{a},\hat{b},\hat{\beta}),$
where $I_0^{-1}(\hat{\alpha},\hat{a},\hat{b},\hat{\beta})$ is the inverse of the observed information matrix which defined as follows
\begin{eqnarray}
I_0^{-1}  & = & - \left[
\begin{array}{cccc}
\frac{\partial ^{2}}{\partial \alpha ^{2}} & \frac{\partial ^{2}}{\partial
\alpha \partial a} & \frac{\partial ^{2}}{\partial \alpha \partial b} &
\frac{\partial ^{2}}{\partial \alpha \partial \beta }
\\
\frac{\partial ^{2}}{\partial a\partial \alpha } & \frac{\partial ^{2}}{%
\partial a^{2}} & \frac{\partial ^{2}}{\partial a\partial b} & \frac{%
\partial ^{2}}{\partial a\partial \beta }
\\
\frac{\partial ^{2}}{\partial b\partial \alpha } & \frac{\partial ^{2}}{%
\partial b\partial a} & \frac{\partial ^{2}}{\partial b^{2}} & \frac{%
\partial ^{2}}{\partial b\partial \beta }
\\
\frac{\partial ^{2}}{\partial \beta \partial \alpha } & \frac{\partial ^{2}}{\partial \beta \partial a} & \frac{\partial ^{2}}{\partial \beta \partial b} & \frac{\partial ^{2}}{\partial \beta ^{2}}
\end{array}
\right]^{-1} 
\nonumber\\ \label{5-8}
 & = &  \left[
\begin{array}{cccc}
Var(\hat{\alpha}) & cov(\hat{\alpha},\hat{a}) & cov(\hat{\alpha},\hat{b}) &
cov(\hat{\alpha},\hat{\beta})
\\
cov(\hat{a},\hat{\alpha}) & Var(\hat{a}) & cov(\hat{a},\hat{b}) & cov(\hat{a},\hat{\beta})
\\
cov(\hat{b},\hat{\alpha}) & cov(\hat{b},\hat{a}) & Var(\hat{b}) & cov(\hat{b},\hat{\beta})
\\
cov(\hat{\beta},\hat{\alpha}) & cov(\hat{\beta},\hat{a}) & cov(\hat{\beta},
\hat{b}) & Var(\hat{\beta})
\end{array}
\right].
\end{eqnarray}

\noindent
The second partial derivatives included in $I_0^{-1}$ are given as follows
\begin{eqnarray*}
\frac{\partial ^{2}L}{\partial \beta ^{2}} & = & \frac{-n}{\beta ^{2}},
\hspace{3.6cm}
\frac{\partial ^{2}L}{\partial \beta \partial \alpha }=\sum^n_{i=1}\frac{\varphi (x_{i},a,b)-1}{\psi (x_{i},\alpha ,a,b)},
\\
\frac{\partial ^{2}L}{\partial \beta \partial a} & = & \alpha \sum^n_{i=1}\frac{x_{i}\varphi (x_{i},a,b)}{\psi (x_{i},\alpha ,a,b)},%
\hspace{1.5cm}
\frac{\partial ^{2}L}{\partial \beta \partial b}=\frac{\alpha }{2 }\sum^n_{i=1} \frac{x_{i}^{2}\varphi (x_{i},a,b)}{\psi
(x_{i},\alpha ,a,b)},
\\
\frac{\partial ^{2}L}{\partial \alpha ^{2}}& = &\frac{-n}{\alpha ^{2}}-(\beta -1)\sum^n_{i=1} \frac{\left[ \varphi (x_{i},a,b)-1\right]
^{2}\left[ \psi (x_{i},\alpha ,a,b)+1\right] }{\left[ \psi (x_{i},\alpha
,a,b)\right] ^{2}},
\\
\frac{\partial ^{2}L}{\partial \alpha \partial a}& =&-\sum^n_{i=1} x_{i}\varphi (x_{i},a,b)+(\beta -1)\sum^n_{i=1} \frac{x_{i}\varphi (x_{i},a,b)h(x_{i},\alpha ,a,b)}{\left[ \psi
(x_{i},\alpha ,a,b)\right] ^{2}},
\\
\frac{\partial ^{2}L}{\partial \alpha \partial b} &= &-\frac{1}{2}\sum^n_{i=1} x_{i}^{2}\varphi (x_{i},a,b)+\frac{(\beta -1)}{2} \sum^n_{i=1}\frac{x_{i}^{2}\varphi (x_{i},a,b)h(x_{i},\alpha ,a,b)}{\left[ \psi (x_{i},\alpha ,a,b)\right] ^{2}},
\\
\frac{\partial ^{2}L}{\partial a^{2}} & = & -\sum^n_{i=1}
\frac{1}{\left( a+bx_{i}\right) ^{2}}-\alpha \sum^n_{i=1} x_{i}^{2}\varphi (x_{i},a,b)+(\beta -1)\alpha \times
\\
& &
\hspace{3cm} \sum^n_{i=1} \frac{x_{i}^{2}\varphi (x_{i},a,b)\tau (x_{i},\alpha ,a,b)}{\left[ \psi (x_{i},\alpha ,a,b)\right] ^{2}},
\\
\frac{\partial ^{2}L}{\partial a\partial b}& = &-\sum^n_{i=1} \frac{x_{i}}{\left( a+bx_{i}\right) ^{2}}-\frac{\alpha }{2}\sum^n_{i=1} x_{i}^{3}\varphi (x_{i},a,b)+\frac{(\beta -1)\alpha }{2}
\times
\\
&&
\hspace{3cm}
\sum^n_{i=1}\frac{x_{i}^{3}\varphi
(x_{i},a,b)\tau (x_{i},\alpha ,a,b)}{\left[ \psi (x_{i},\alpha ,a,b)\right]
^{2}},
\\
\frac{\partial ^{2}L}{\partial b^{2}} & = &-\sum^n_{i=1}
\frac{x_{i}^{2}}{\left( a+bx_{i}\right) ^{2}}-\frac{\alpha }{4}\sum^n_{i=1}x_{i}^{4}\varphi (x_{i},a,b)+\frac{(\beta -1)\alpha }{4}
\\
&&
\hspace{3cm}
\sum^n_{i=1}\frac{x_{i}^{4}\varphi (x_{i},a,b)\tau
(x_{i},\alpha ,a,b)}{\left[ \psi (x_{i},\alpha ,a,b)\right] ^{2}},
\end{eqnarray*}

\noindent
where the nonlinear functions $\psi (x_i,\alpha ,a,b),$ $\varphi(x_i,a,b),$ $h(x_i,\alpha ,a,b)$ and $\tau (x_i,\alpha ,a,b)$ are given by
\begin{equation*}
\varphi (x_{i},a,b)=e^{ax_{i}+\frac{b}{2}x_{i}^{2}},\qquad  \psi
(x_{i},\alpha ,a,b)=e^{\alpha \left( e^{ax_{i}+\frac{b}{2}x_{i}^{2}}-1\right)}-1,
\end{equation*}
\begin{equation*}
h(x_{i},\alpha ,a,b)=e^{\alpha \left( e^{ax_i+\frac{b}{2}x_i^{2}}-1 \right) }\left[1-\alpha \left( e^{ax_{i}+\frac{b}{2}x_i^{2}}-1\right) \right]-1,
\end{equation*}
\begin{equation*}
\tau (x_{i},\alpha ,a,b)=e^{\alpha \left(e^{ax_i+\frac{b}{2}x_i^2}-1
\right) }\left[1-\alpha e^{ax_i+\frac{b}{2}x_i^2}\right]-1.
\end{equation*}

\noindent
The above approach is used to derive the $(1-\delta) 100\%$ confidence intervals for the parameters $\alpha, a,b $ and $\beta $  as in the following forms
\begin{equation*}
\hat{\alpha}\pm Z_{\frac{\delta }{2}}\sqrt{Var(\hat{\alpha})}, \; \hat{a}\pm Z_{\frac{\delta }{2}}\sqrt{Var(\hat{a})}, \; \hat{b}\pm Z_{\frac{\delta }{2}}\sqrt{Var(\hat{b})}, \; \hat{\beta}\pm Z_{\frac{\delta }{2}}\sqrt{Var(\hat{\beta})},
\end{equation*}

\noindent
where $Z_{\frac{\delta }{2}}$ is the upper ($\frac{\delta}{2}$)th percentile of the standard normal distribution.

\section{Data Analysis}
Now we use a real data set to show that the OGE-LFR distribution can be a better model, comparing with many known distributions such as the Exponential(E), Generalized Exponential(GE), Linear Failure Rate(LFR), New Generalized Linear Exponential (NGLE) and Transmuted Linear Exponential (TLE). Consider the data have been obtained from Aarset \cite{1}, and widely reported in many literatures. It represents the lifetimes of 50 devices, and also, possess a bathtub-shaped failure rate property, Table 1.

\begin{center}
Table 1: The data from Aarset \cite{1}.\\
\begin{tabular}{rrrrrrrrrrrrrr} \hline
0.1 & 0.2 & 1 & 1 & 1 & 1 & 1 & 2 & 3 & 6 & 7 & 11 & 12 & 18 \\
18 & 18 & 18 & 18 & 21 & 32 & 36 & 40 & 45 & 46 & 47 & 50 & 55 & 60 \\
63 & 63 & 67 & 67 & 67 & 67  &  72 & 75 & 79 & 82 & 82 & 83 & 84 & 84 \\
84 & 85 & 85  & 85 & 85 & 85 & 86 &  86\\
\hline
\end{tabular}
\end{center}

\noindent
The MLEs of the unknown parameters and the corresponding Kolmogorov--Smirnov(K--S) test statistic for the six models are given in Table 2. 

\begin{center}
Table 2: The MLES of the parameters, the K--S values and p-values.\\
\begin{tabular}{llcc} \hline
The  model & MLE of the parameters &  K--S & P-value(K-S)\\ \hline
E & $\hat{\alpha}$ = 0.0219 &  0.1911  & 0.0519 \\  \hline
GE & $\hat{\alpha}$ = 0.0212, $\hat{\beta}$ = 0.9012 & 0.1940 &  0.0514 \\  \hline
LFR &  $\hat{a}$ = 0.014, $\hat{b}$ = 2.4$\times 10^{-4}$ & 0.1955 & 0.0370
\\ \hline
NGLE &
\begin{tabular}{l}
$\hat{a}$ = 0.0012, $\hat{b}$ = 0.0127, \\
$\hat{c}$ =1.0682, $\hat{\beta}$ = 0.7231\\
\end{tabular}  & 0.2030 & 0.0276 \\
 \hline
TLE &
\begin{tabular}{l}
$\hat{a}$ = 0.0145, $\hat{b}$ = 2.4186 $\times 10^{-4},$ \\
 $\hat{\lambda}$ = -0.0948
 \end{tabular}  & 0.1740 & 0.0855 \\  \hline
OGE-LFR &
\begin{tabular}{l}
$\hat{\alpha}$ = 472.404, $\hat{a}$ = 8.218 $\times 10^{-6}$,
\\
$\hat{b}$ = 6.427 $\times 10^{-7},$ $\hat{\beta}$ = 0.529
\end{tabular} & 0.1627 & 0.12830 \\ \hline
\end{tabular}
\end{center}

\noindent
The values of the log-likelihood functions (-L), AIC (Akaike Information Criterion), the statistics AICC (Akaike Information Citerion with correction), BIC (Bayesian Information Criterion) and HQIC (Hannan-Quinn information criterion) are calculated in Table 3 for the six distributions in order to verify which distribution fits better to these data.

\noindent
\begin{center}
Table 3: The --L, AIC, AICC, BIC and HQIC for devices data.\\
\begin{tabular}{llllll} \hline
The model & --L & AIC & AICC & BIC & HQAIC \\  \hline
E & 241.090 & 484.1792 & 484.2625 & 486.0912 & 484.908 \\
GE & 240.3855 & 484.7710 & 485.0264 & 488.5951 & 486.227 \\
LFR & 238.064 & 480.128 & 480.383 & 483.952 & 481.584 \\
NGLE & 239.49 & 486.98 & 487.869 & 494.6281 & 489.892 \\
TLE & 238.01 & 482.02 & 482.54 & 487.756 & 484.204 \\
OGE-LFR & 232.865 & 473.730 & 474.618 & 481.378 & 476.642 \\ \hline
\end{tabular}
\end{center}

\noindent
Based on Tables 2 and 3, it is shown that OGE-LFR($\alpha, a, b, \beta $) model provide better fit to the data rather than other distributions which we compared with because it has the smallest value of (K-S), AIC, AICC, BIC and HQIC test.\\

\noindent
To show that the likelihood equation have unique solution, we plot the profiles of the log-likelihood function of $\alpha, a, b$ and $\beta $ in Figures 5-6.
\begin{figure}[htp]
\centering
\includegraphics[width=0.37\textwidth]{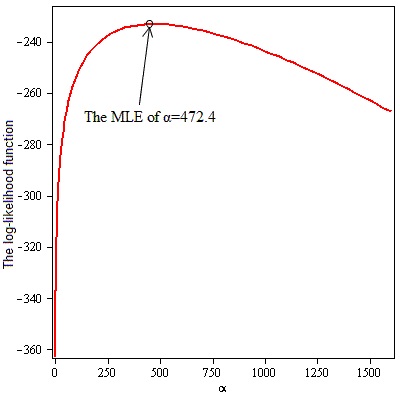}
\includegraphics[width=0.4\textwidth]{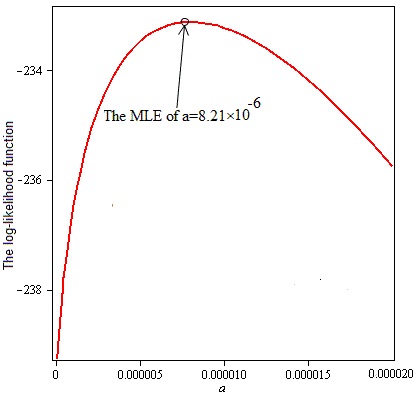}
\caption{The profile of the log-likelihood function of $\alpha, a$.}
\label{fig5}
\end{figure}

\begin{figure}[htp]
\centering
\includegraphics[width=0.37\textwidth]{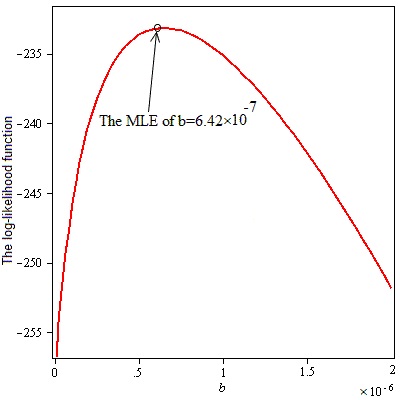}
\includegraphics[width=0.37\textwidth]{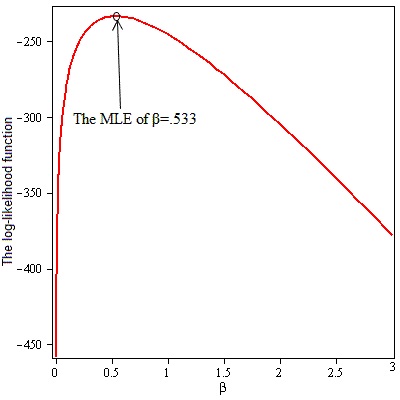}
\caption{The profile of the log-likelihood function of $b, \beta$.}
\label{fig6}
\end{figure}

\noindent
The nonparametric estimate of the survival function using the Kaplan-Meier method and its fitted parametric estimations when the distributions is assumed to be E, GE, LFR, NGLE, TLE and OGE-LFR are computed and plotted in Figure 7.
\begin{figure}[htp]
\centering
\includegraphics[width=0.5\textwidth]{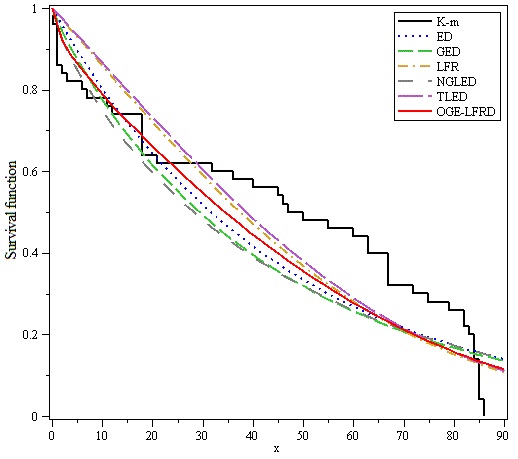}
\caption{The Kaplan-Meier estimate of survival function and fitted survival functions.}
\label{fig7}
\end{figure}

\noindent
Figures 8 and 9, give the form of the probability density functions and the hazard functions for the ED, GED, LFRD, NGLED, TLED, OGE-LFRD distributions which are used to fit the data after replacing the unknown parameters included in each distribution by their MLE.

\begin{figure}[htp]
\centering
\includegraphics[width=0.5\textwidth]{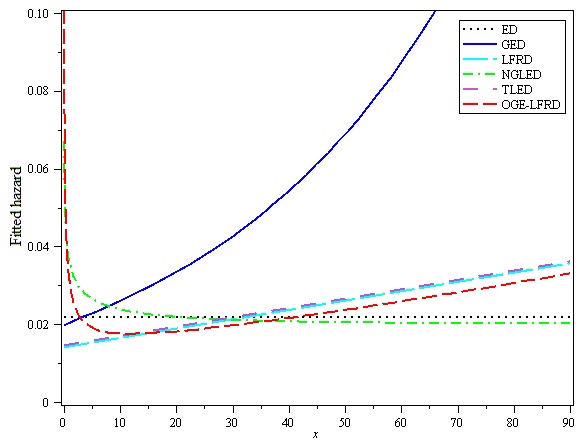}
\caption{The Fitted hazard functions for the data.}
\label{fig7}
\end{figure}

\begin{figure}[htp]
\centering
\includegraphics[width=0.5\textwidth]{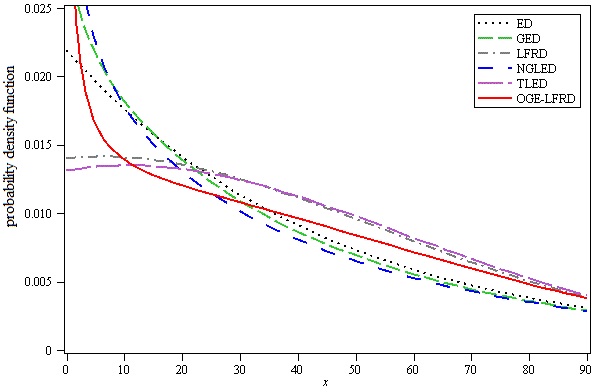}
\caption{The Fitted probability density functions for the data.}
\label{fig8}
\end{figure}

\newpage
\section{Conclusions}
In this article, we have introduced a new four-parameter model called odd generalized exponential linear failure rate distribution. We refer to the new model as the OGE-LFR distribution and study some of its mathematical and statistical properties. We provide the pdf, the cdf, the hazard rate function and the reversed hazard function for the new model also we provide an explicit expression for the moments. The model parameters are estimated by maximum likelihood method. We use application on set of real data to compare the OGE-LFR with other known distributions such as Exponential (E), Generalized Exponential (GE), Linear Failure Rate (LFR), New Generalized Linear Exponential (NGLE) and Transmuted Linear Exponential (TLE). Applications on set of real data showed that the OGE-LFR is the best distribution for fitting these data sets compared with ED, GED, LFRD, NGLED and TLED distributions.

\end{document}